\documentclass[12pt, leqno]{article}
\usepackage{amsmath,amsfonts,amsthm,amscd,amssymb,verbatim} 
\usepackage[all]{xy}

\numberwithin{equation}{section}

\theoremstyle{plain}
\newtheorem{theorem}{Theorem}[section]

\newtheorem{proposition}[theorem]{Proposition}

\newtheorem{corollary}[theorem]{Corollary}

\theoremstyle{definition}
\newtheorem{remark}[theorem]{Remark}

\newcommand{\beast}{\begin{eqnarray*}}
\newcommand{\east}{\end{eqnarray*}}

\newcommand{\N}{{\Bbb N}}
\newcommand{\Z}{{\Bbb Z}}

\newcommand{\Q}{{\Bbb Q}}
\newcommand{\R}{{\Bbb R}}

\newcommand{\fG}{\wh{{\Bbb G}}}

\newcommand{\F}{{\Bbb F}}

\newcommand{\Af}{{\Bbb A}}
\newcommand{\fAf}{\wh{{\Bbb A}}}

\newcommand{\Spec}{{\mathrm{Spec}}\,}

\newcommand{\Spf}{{\mathrm{Spf}}\,}

\newcommand{\lra}{\longrightarrow}
\newcommand{\ra}{\rightarrow}
\newcommand{\hra}{\hookrightarrow}

\newcommand{\lla}{\longleftarrow}

\newcommand{\ti}[1]{\widetilde{#1}}

\newcommand{\ul}[1]{\underline{#1}}
\newcommand{\ol}[1]{\overline{#1}}
\newcommand{\os}{\overset}

\newcommand{\Hom}{{\mathrm{Hom}}}

\newcommand{\Ker}{{\mathrm{Ker}}}

\newcommand{\Rep}{{\mathrm{Rep}}}

\newcommand{\cE}{{\cal E}}
\newcommand{\cF}{{\cal F}}

\newcommand{\cO}{{\cal O}}

\newcommand{\cU}{{\cal U}}

\newcommand{\cX}{{\cal X}}

\newcommand{\fU}{{\frak U}}

\renewcommand{\sp}{{\mathrm{sp}}}

\newcommand{\wh}{\widehat}

\newcommand{\pa}{\partial}

\renewcommand{\1}{{\bold 1}}

\newcommand{\MIC}{{\mathrm{MIC}}}

\newcommand{\Isoc}{{\mathrm{Isoc}}}
\newcommand{\Isocd}{\Isoc^{\dagger}}
\newcommand{\FIsoc}{F\text{-}\Isoc}
\newcommand{\FIsocd}{F\text{-}\Isocd}
\newcommand{\FMIC}{F\text{-}\MIC}

\newcommand{\codim}{{\mathrm{codim}}}

\begin{document}
\title{Purity for overconvergence}
\author{Atsushi Shiho
\footnote{
Graduate School of Mathematical Sciences, 
University of Tokyo, 3-8-1 Komaba, Meguro-ku, Tokyo 153-8914, JAPAN. 
E-mail address: shiho@ms.u-tokyo.ac.jp \, 
Mathematics Subject Classification (2000): 12H25, 14F35.}}
\date{}
\maketitle

\begin{abstract}
Let $X \hra \ol{X}$ be an open immersion of smooth varieties over a 
field of characteristic $p>0$ such that the complement is 
a simple normal crossing divisor and let $\ol{Z} \subseteq 
Z \subseteq \ol{X}$ be closed 
subschemes of codimension at least $2$. 
In this paper, we prove that the 
canonical restriction functor between the category of overconvergent 
$F$-isocrystals $\FIsocd(X,\ol{X}) \lra \FIsocd(X \setminus Z, 
\ol{X} \setminus \ol{Z})$ is an equivalence of categories. We also prove an 
application to the category of $p$-adic representations of the fundamental 
group of $X$, which is a 
higher-dimensional version of a result of Tsuzuki.  
\end{abstract}

\tableofcontents

\section*{Introduction}
Let $X$ be a regular scheme and let $Z \subseteq X$ be a closed subscheme 
of codimension at least $2$. Then, by the famous Zariski-Nagata purity, 
any locally constant constructible sheaf on $(X \setminus Z)_{\rm et}$ 
extends uniquely to a locally constant constructible sheaf on 
$X_{\rm et}$. As a $p$-adic analogue of this fact, Kedlaya proved 
in \cite[5.3.3]{kedlayaI}
the following result on the purity for overconvergent isocrystals: 
Let $k$ be a field of characteristic $p>0$, let $X \hra \ol{X}$ be an 
open immersion of $k$-varieties with $X$ smooth and let $Z \subseteq X$ 
be a closed subscheme of $X$ of codimension at least $2$ such that 
$X \setminus Z$ is dense in $\ol{X}$. Then the 
restriction functor of the categories of overconvergent isocrystals 
$$ \Isocd(X,\ol{X}) \lra \Isocd(X \setminus Z, \ol{X}) $$
is an equivalence of categories. \par 
In this paper, we prove a slight generalization of the result of Kedlaya in 
the case where $\ol{X}$ is smooth, $\ol{X} \setminus X$ is a simple 
normal crossing divisor and the overconvergent isocrystals are endowed 
with Frobenius structure. The precise statement is as follows: 
Let $X \hra \ol{X}$ be an open immersion of smooth varieties over a 
field of characteristic $p>0$ such that the complement is 
a simple normal crossing divisor and let $\ol{Z} \subseteq 
Z \subseteq \ol{X}$ be closed 
subschemes of codimension at least $2$. Then the 
canonical restriction functor between the category of overconvergent 
$F$-isocrystals 
$$\FIsocd(X,\ol{X}) \lra \FIsocd(X \setminus Z, 
\ol{X} \setminus \ol{Z})$$ 
is an equivalence of categories. In the case $\ol{Z} = \emptyset$, 
this is reduced to the above-mentioned result of Kedlaya. A new point in our 
result is that we have the purity also on the second factor 
of the pair 
$(X,\ol{X})$ (the overconvergent locus). 
In this sense, we can say our result as `the purity for 
overconvergence'. \par 
As an application, we consider the following. 
Assume that $k$ is perfect and that $X$ is connected. 
Then Crew proved the equivalence of categories 
$$G: \Rep_{K^{\sigma}}(\pi(X)) \os{=}{\lra} \FIsoc(X)^{\circ}$$ 
between the category $\Rep_{K^{\sigma}}(\pi(X))$ of 
finite-dimensional continuous representations of the fundamental 
group $\pi_1(X)$ of $X$ over $K^{\sigma}$ 
(where $K$ is a complete discrete valuation field with residue 
field $k$ endowed with a lifting $\sigma$ of Frobenius and $K^{\sigma}$ 
denotes the fixed field of $\sigma$) and the category 
$\FIsoc(X)^{\circ}$ of unit-root convergent $F$-isocrystals on $X$ over 
$K$. When $X$ is a curve, Tsuzuki proved in \cite{tsuzukicurve} that 
the subcategory $\FIsocd(X,\ol{X})^{\circ}$ of $\FIsoc(X)^{\circ}$
 consisting of unit-root overconvergent $F$-isocrystals on $(X,\ol{X})$ over 
$K$ corresponds (via $G$) to the subcategory of 
$\Rep_{K^{\sigma}}(\pi(X))$ consisting of `the representations with 
finite local monodromy'. A generalization of this result to 
higher-dimensional case is proved by Kedlaya in 
\cite[2.3.7, 2.3.9]{kedlayaswanII} 
based on a result of Tsuzuki in \cite{tsuzuki}. One 
drawback in Kedlaya's result 
is that the number of valuations which we should look at is infinite.  
In this paper, we give an alternative definition of `the representations with 
finite local monodromy' which looks at only finitely many 
discrete valuations of $k(X)$ and prove that the category of such 
representations is 
still equivalent (via $G$) to $\FIsoc(X,\ol{X})^{\circ}$, 
by using the purity for overconvergence mentioned in the previous 
paragraph. \par 
The content of each section is as follows: In the first section, 
we give several notations, conventions and terminologies 
which we often use in this paper. In the second section, 
we prove a theorem on the relation of generic convergence and 
overconvergence of $p$-adic differential equations in certain case, 
which we need to prove the purity for overconvergence. 
In the third section, we give a proof of the purity for overconvegergence. 
In the fourth section, we explain the above-mentioned 
application to $p$-adic representations. \par 
The author is grateful to Nobuo Tsuzuki for giving me an opportunity to 
give a talk on the topic of this paper at Tohoku University. 
The author is partly supported by Grant-in-Aid for Young Scientists (B) 
21740003 from 
the Ministry of Education, Culture, Sports, Science and Technology, Japan 
and 
Grant-in-Aid for Scientific Research (B) 22340001 
(representative: Nobuo Tsuzuki) from 
Japan Society for the Promotion of Science. 

\section{Preliminaries}
In this section, we give several notations, conventions and terminologies 
which we often use in this paper. \par 
Throughout this paper, $K$ is 
a complete discrete valuation field of mixed characteristic 
$(0,p)$ with ring of integers $O_K$ and residue field $k$. 
Let $q$ be a fixed power of $p$ and assume that $k$ contains $\F_q$. 
Moreover, let $\sigma: O_K \lra O_K$ be an endomorphism of $O_K$ 
which lifts the $q$-th power map on $k$. We denote the endomorphism 
on $K$ induced by $\sigma$ by the same symbol. 
Let $|\cdot|: K \lra \R_{\geq 0}$ be a fixed valuation of $K$ and 
let $\Gamma^*$ be $\sqrt{|K^{\times}|} \cup \{0\}$. \par 
All the schemes appearing in this paper are assumed to be 
separated of finite type 
over $k$ and all the $p$-adic formal schemes appearing in this paper are
 assumed to be separated, topologically of finite type over $\Spf O_K$. 
For a $p$-adic formal scheme $\cX$, we call the scheme $\cX \otimes_{O_K} k$ 
the special fiber of $\cX$. For a smooth scheme $X$ over $k$, a lift of 
$X$ is a closed immersion $X \hra \cX$ into a smooth $p$-adic formal 
scheme $\cX$ such that $X$ is naturally isomorphic to the special fiber 
of $\cX$. (Note that, when $X$ is affine and smooth over $k$, 
a lift of $X$ always exists.) For a $p$-adic formal scheme $\cX$, 
we denote the associated rigid space over $K$ by $\cX_K$. Then we have the 
specialization map $\sp: \cX_K \lra \cX$. \par 
A pair $(X,\ol{X})$ is a pair of schemes $X,\ol{X}$ (separated of finite 
type over $k$) endowed with an open immersion $X \hra \ol{X}$ over $k$. 
A smooth pair is a pair $(X,\ol{X})$ such that $X, \ol{X}$ are smooth and 
that $\ol{X} \setminus X$ (endowed with the reduced closed subscheme 
structure) is a simple normal crossing 
divisor in $\ol{X}$. 
A formal pair $(\cX,\ol{\cX})$ is a pair of $p$-adic formal 
schemes $\cX,\ol{\cX}$ (separated, topologically of finite 
type over $\Spf O_K$) endowed with an open immersion $\cX \hra \ol{\cX}$ over 
$\Spf O_K$. 
A formal smooth pair is a 
pair $(\cX,\ol{\cX})$ such that $\cX, \ol{\cX}$ are smooth and 
that $\ol{\cX} \setminus \cX$ (with some closed formal subscheme structure) 
is a relative simple normal crossing 
divisor in $\ol{\cX}$. A morphism of (formal) pairs $f:(X,\ol{X}) \lra 
(Y,\ol{Y})$ is defined as the morphism $f:\ol{X} \lra \ol{Y}$ satisfying 
$f(X) \subseteq Y$. It is called strict if $f^{-1}(Y) = X$ and it is called 
finite etale or a closed immersion if so is $f:\ol{X} \lra \ol{Y}$. \par 
For a formal pair $(\cX, \ol{\cX})$, we call the pair 
$(\cX \otimes_{O_K} k, \ol{\cX} \otimes_{O_K} k)$ 
the special fiber of $(\cX, \ol{\cX})$. 
For a smooth pair $(X, \ol{X})$, a lift of 
$(X, \ol{X})$ is a strict closed immersion 
$(X, \ol{X}) \hra (\cX, \ol{\cX})$ into a formal smooth pair 
such that $(X, \ol{X})$ is naturally isomorphic to the special fiber 
of $(\cX, \ol{\cX})$. \par 
Let $(\cX,\ol{\cX})$ be a formal smooth pair.  
Then a strict neighborhood 
of $\cX_K$ in $\ol{\cX}_K$ is an admissible open set $V$ of $\ol{\cX}_K$ 
such that $\{V, \sp^{-1}(\ol{\cX} \setminus \cX)\}$ forms an admissible 
covering of $\ol{\cX}_K$. For strict neighborhoods $V, W$ 
of $\cX_K$ in $\ol{\cX}_K$ with $W \subseteq V$, we denote the 
canonical open immersion $W \hra V$ by $\alpha_{WV}$ and using this, 
we define the sheaf of overconvergent sections $j^{\dagger}\cO_{\ol{\cX}_K}$ 
by $j^{\dagger}\cO_{\ol{\cX}_K} := 
\varinjlim_{V} \alpha_{V\ol{\cX}_K,*}\cO_V$. For a strict neighborhood $V$ 
of $\cX_K$ in $\ol{\cX}_K$ and an $\cO_V$-module $E$, we put 
$j_V^{\dagger}E := \varinjlim_{W \subseteq V} 
\alpha_{W\ol{\cX}_K,*}\alpha^{-1}_{WV}E$. 
Let $\MIC(\cX_K,\ol{\cX}_K)$ be the category of 
pairs $(V,(E,\nabla))$ consisting of a strict neighborhood 
$V$ of $\cX_K$ in $\ol{\cX}_K$ and a $\nabla$-module
($=$locally free module of finite rank endowed with an 
integrable connection) $(E,\nabla)$ on $V$ over $K$, whose set of 
morphisms is defined by $\Hom((V,(E,\nabla)),(V',(E',\nabla'))) 
:= \varinjlim_{V''}\Hom((E,\nabla)|_{V''},(E',\nabla')|_{V''})$, 
where $V''$ runs through strict neighborhoods of $\cX_K$ in 
$\ol{\cX}_K$ contained in $V\cap V'$. We 
call an object in $\MIC(\cX_K,\ol{\cX}_K)$ a $\nabla$-module on 
a strict neighborhood of $\cX_K$ in $\ol{\cX}_K$ by abuse of 
terminology, and we will often denote it simply by $(E,\nabla)$ in the 
following. On the other hand, let $\MIC(j^{\dagger}\cO_{\ol{\cX}_K})$ be
 the category of locally free $j^{\dagger}\cO_{\ol{\cX}_K}$-modules $E$
of finite rank endowed with an integrable connection of the form 
$\nabla: E \lra E \otimes j^{\dagger}_{\ol{\cX}_K}\Omega^1_{\ol{\cX}_K}$. 
Then we have the equivalence of categories 
\cite[2.1.10, 2.2.3]{berthelotrig}
$$ \MIC(\cX_K,\ol{\cX}_K) \os{=}{\lra} \MIC(j^{\dagger}\cO_{\ol{\cX}_K}); 
\quad (V, (E,\nabla)) \mapsto j_V^{\dagger}(E, \nabla). $$
When $\cX_K = \ol{\cX}_K$, we denote the category 
$\MIC(\cX_K,\ol{\cX}_K)$ simply by $\MIC(\cX_K)$. \par 
On the other hand, when we are given a pair $(X,\ol{X})$, 
there exists a notion of overconvergent isocrystals on $(X,\ol{X})$ over $K$: 
Let us denote the category of overconvergent 
isocrystals on $(X,\ol{X})$ over $K$ by $\Isocd(X,\ol{X})$. When 
$X=\ol{X}$, we call this category the category of convergent isocrystals 
on $X$ over $K$ and denote it by $\Isoc(X)$. \par 
Suppose that 
$(X,\ol{X})$ is a smooth pair which admits a lift $(X, \ol{X}) \hra 
(\cX, \ol{\cX})$. 
Then we have the 
canonical fully faithful functor (called the functor of realization)
$$ \Phi_{(\cX,\ol{\cX})}: \Isocd(X,\ol{X}) \lra \MIC(\cX_K,\ol{\cX}_K) = 
\MIC(j^{\dagger}\cO_{\ol{\cX}_K}). $$
An object in $\MIC(\cX_K,\ol{\cX}_K) = 
\MIC(j^{\dagger}\cO_{\ol{\cX}_K})$ is called overconvergent if 
it is in the essential image of $\Phi_{(\cX,\ol{\cX})}$. 
When $X=\ol{X}$ (hence $\cX = \ol{\cX}$), we denote the functor 
$\Phi_{(\cX,\ol{\cX})}$ by 
$$ \Phi_{\cX}: \Isoc(X) \lra \MIC(\cX_K), $$
and an object in $\MIC(\cX_K)$ is called convergent if 
it is in the essential image of $\Phi_{\cX}$. \par 
For a pair $(X,\ol{X})$, the $q$-th power Frobenius map 
$F: (X,\ol{X}) \lra (X,\ol{X})$ is a morphism over $\sigma^*: \Spf O_K \lra 
\Spf O_K$ and so it induces a $\sigma$-linear functor 
$F^*: \Isocd(X,\ol{X}) \lra \Isocd(X,\ol{X})$. 
An overconvergent $F$-isocrystal on $(X,\ol{X})$ over $K$ is a pair 
$(\cE,\Psi)$, where $\cE \in \Isoc(X,\ol{X})$ and $\Psi$ is an 
isomorphism $F^*\cE \os{=}{\lra} \cE$. We denote the category of 
overconvergent $F$-isocrystals on $(X,\ol{X})$ over $K$ 
by $\FIsocd(X,\ol{X})$. 
When $X=\ol{X}$, we call this category the category of 
convergent $F$-isocrystals on $X$ over $K$ 
and denote it by $\FIsoc(X)$. \par 
Suppose that $(X,\ol{X})$ is a smooth pair which admits 
a lift $(X, \ol{X}) \hra (\cX, \ol{\cX})$ 
and an endomorphism $F: \ol{\cX} \lra \ol{\cX}$ over $\sigma^*$ 
lifting the $q$-th power Frobenius on $\ol{X}$. 
Then $F$ induces an endomorphism on the formal smooth pair 
$(\cX,\ol{\cX})$ and an 
endomorphism on $\cX_K$ and on $\ol{\cX}_K$, which we denote also by $F$. 
Let 
$\FMIC(\cX_K,\ol{\cX}_K)$ be the category of pairs 
$((V,(E,\nabla)),\Psi)$, where 
$(V,(E,\nabla)) \in \MIC(\cX_K,\ol{\cX}_K)$ 
and $\Psi$ is an isomorphism 
$(F^{-1}(V), (F|_{F^{-1}(V)})^*(E,\nabla)) \os{=}{\lra} (V,(E,\nabla))$ in 
$\MIC(\cX_K,\ol{\cX}_K)$. Also, $F$ induces the homomorphism 
$F^*j^{\dagger}\cO_{\ol{\cX}_K} \lra j^{\dagger}\cO_{\ol{\cX}_K}$ 
and so we can define the category 
$\FMIC(j^{\dagger}\cO_{\ol{\cX}_K})$ as the category of pairs 
$((E,\nabla),\Psi)$ consisting of $(E,\nabla) \in 
\MIC(j^{\dagger}\cO_{\ol{\cX}_K})$ and an isomorphism 
$\Psi: F^*(E,\nabla) \os{=}{\lra} (E,\nabla)$. Then we have an 
equivalence of categories 
$$ \FMIC(\cX_K,\ol{\cX}_K) \os{=}{\lra} 
\FMIC(j^{\dagger}\cO_{\ol{\cX}_K}); 
\quad ((V, (E,\nabla)),\Psi) \mapsto 
(j_V^{\dagger}(E, \nabla),j_{V''}^{\dagger}\Psi) $$
(where $V''$ is a strict neighborhood on which $\Psi$ is defined), and 
$\Phi_{(\cX,\ol{\cX})}$ induces the fully faithful functor 
$$ \Phi_{(\cX,\ol{\cX})}: \FIsocd(X,\ol{X}) \lra 
\FMIC(\cX_K,\ol{\cX}_K) = 
\FMIC(j^{\dagger}\cO_{\ol{\cX}_K}). $$
When $X=\ol{X}$ (hence $\cX = \ol{\cX}$), we denote it by 
$$ \Phi_{\cX}: \FIsoc(X) \lra \FMIC(\cX_K). $$
Finally, we prepare a notation to express $p$-adic polyannulus: 
For a closed interval $[a,b]$ with $a,b \in [0,1] \cap \Gamma^*$, 
we define the 
polyannulus $A^n_K[a,b]$ by 
$$ A^n_K[a,b] := \{x \in (\fAf^{n}_{O_K})_{K} \,|\, \forall i, 
|t_i(x)| \in [a,b] \} $$
(where $t_i$ is the $i$-th coordinate of $(\fAf^{n}_{O_K})_{K}$). 

\begin{remark}
We explained several notions only in the restricted cases: 
For example, the functor of realization can be defined for more 
general situation. 
For details on overconveregent isocrystals and $\nabla$-modules, 
see \cite{berthelotrig} and \cite[\S 2]{kedlayaI}, for example. 
\end{remark}

\section{Generic convergence and overconvergence}

In this section, we prove the following theorem, 
which claims that `generic convergence implies overconvergence' in 
certain situation. 

\begin{theorem}\label{coc}
Let $(\cX,\ol{\cX})$ be a formal smooth pair. Let $\cU$ be an open dense 
$p$-adic formal subscheme of $\cX$, denote the canonical morphism 
of formal smooth pairs $(\cU,\cU) \hra (\cX,\ol{\cX})$ by $\alpha_{\cU}$ and 
denote the induced functor 
$\MIC(\cX_K,\ol{\cX}_K) \lra \MIC(\cU_K)$ by $\alpha_{\cU}^*$. 
Then an object $(E,\nabla)$ in $\MIC(\cX_K,\ol{\cX}_K)$ is 
overconvergent if and only if $\alpha_{\cU}^*(E,\nabla)$ is 
convergent. 
\end{theorem}

\begin{proof}
The proof is divided into several steps. \par 
\ul{Step 1}: First we prove the `only if' part.
Let $(X,\ol{X})$ be the special fiber of $(\cX,\ol{\cX})$ and let 
$U$ be the special fiber of $\cU$. Then the `only if' part 
follows from the following 
commutative 
diagram, which is the functoriality of the functors of the form 
$\Phi_{(?,?)}$: 
\begin{equation*}
\begin{CD}
\Isocd(X,\ol{X}) @>{\Phi_{(\cX,\ol{\cX})}}>> \MIC(\cX_K,\ol{\cX}_K) \\ 
@V{\alpha_{\cU}^*}VV @V{\alpha_{\cU}^*}VV \\ 
\Isoc(U) @>{\Phi_{\cU}}>> \MIC(\cU_K). 
\end{CD}
\end{equation*}
(Here the functor $\alpha_{\cU}^*$ on the right is the pull-back 
functor induced by the morphisms $\cU_K \hra \cX_K, \cU_K \hra \ol{\cX}_K$ 
induced by $\alpha_{\cU}$.) \par 
\ul{Step 2}: Here we reduce the proof of the theorem to the case 
$\cU = \cX$. Let us factorize $\alpha_{\cU}$ as 
$$ (\cU,\cU) \lra
(\cX,\cX) \os{\alpha_{\cX}}{\lra} (\cX,\ol{\cX}). $$
Then $\alpha_{\cX}^*(E,\nabla)$ is an object in $\MIC(\cX_K/K)$ whose 
restriction to $\MIC(\cU_K/K)$ is convergent. So, by \cite[2.16]{ogus}, 
$\alpha_{\cX}^*(E,\nabla)$ is convergent. So, if we assume that the 
theorem is true in the case $\cU = \cX$, we can conclude that 
$(E,\nabla)$ is overconvergent. In the following, we will assume that 
$\cU = \cX$. \par 
\ul{Step 3}: In this step, we follow the argument in \cite[2.4]{co} 
and prove that the theorem is true if it is true for 
formal smooth pairs of the form 
$(\cX_0, \ol{\cX}_0) := (\fG_{m,O_K}^n \times \fAf_{O_K}^m, 
\fAf_{O_K}^{n+m})\,(n,m\in\N)$ (and $\cU = \cX_0$). 
First, by claim in the proof of \cite[Proposition 2.4]{co}, 
we have the following: 
For any closed point $x$ of $\ol{X}$, there exist 
open immersions $\ol{\cU}_x \hra \ol{\cX}_x \hra \ol{\cX}$ with 
$x \in \ol{\cU}_x$, 
an open formal subscheme $\cX_x$ of $\ol{\cX}_x$ and a diagram 
of formal smooth pairs 
$$ (\cX,\ol{\cX}) \os{j}{\lla} (\cX_x, \ol{\cX}_x) \os{f}{\lra} 
(\fG_{m,O_K}^n \times \Af_{O_K}^m, \fAf_{O_K}^{n+m}) $$ 
(where $j$ is induced by the open immersion $\ol{\cX}_x \hra 
\ol{\cX}$) for some $n,m$ 
such that $f$ is a strict finite etale morphism and that 
the morphism $(\ol{\cU}_x \cap \cX_x, \ol{\cU}_x) \lra (\ol{\cU}_x \cap \cX, 
\ol{\cU}_x)$ induced by $j$ is an isomorphism. \par 
For any closed point 
$x$ in $\ol{X}$, let us choose 
open immersions $\ol{\cU}_x \hra \ol{\cX}_x \hra \ol{\cX}$ with $x \in 
\ol{\cU}_x$, 
an open formal subscheme $\cX_x$ of $\ol{\cX}_x$ and morphisms $j,f$ as 
above. Then, if 
$(E,\nabla)$ is an object in $\MIC(\cX_K,\ol{\cX}_K)$ such that 
$\alpha_{\cX}^*(E,\nabla)$ is 
convergent, the restriction of it to $\MIC(\cX_{x,K})$ is convergent 
for any $x$. Then, if we succeed to prove that the restriction 
of $(E,\nabla)$ to 
$\MIC(\cX_{x,K},\ol{\cX}_{x,K})$ 
is overconvergent, 
it implies the overconvergence of the restriction of 
$(E,\nabla)$ to 
$\MIC(\ol{\cU}_{x,K} \cap \cX_{x,K},\ol{\cU}_{x,K}) = 
\MIC(\ol{\cU}_{x,K} \cap \cX_{K},\ol{\cU}_{x,K})$ 
for all $x$ and this implies the overconvergence of $(E,\nabla)$ because 
it is known \cite[2.3]{co} 
that the overconvergence condition for an object in 
$\MIC(\cX_K,\ol{\cX}_K)$ has local nature with respect to the Zariski 
topology on $\ol{\cX}$. So, it suffices to prove the theorem for 
$(\cX_x, \ol{\cX}_x)$, that is, we may assume that there exists a 
strict finite etale morphism $f: (\cX, \ol{\cX}) \lra 
(\fG_{m,O_K}^n \times \fAf_{O_K}^m, \fAf_{O_K}^{n+m}) =: 
(\cX_0, \ol{\cX}_0) $ to prove the proposition. \par 
Let $(X_0,\ol{X}_0)$ be the special fiber of $(\cX_0,\ol{\cX}_0)$. Then, by 
the argument in the proof of \cite[2.4]{co} and \cite[5.1]{tsuzuki}, 
we have the push-forward functors 
\begin{align*}
& f_*: \Isocd(X,\ol{X}) \lra \Isocd(X_0,\ol{X}_0), \\ 
& f_*: \MIC(\cX_K,\ol{\cX}_K) \lra \MIC(\cX_{0,K},\ol{\cX}_{0,K})
\end{align*}
which makes the diagram 
\begin{equation}\label{push1}
\begin{CD}
\Isocd(X,\ol{X}) @>{\Phi_{(\cX,\ol{\cX})}}>> 
\MIC(\cX_K,\ol{\cX}_K) \\ 
@V{f_*}VV @V{f_*}VV \\ 
\Isocd(X_0,\ol{X}_0) 
@>{\Phi_{(\cX_0,\ol{\cX}_0)}}>> 
\MIC(\cX_{0,K},\ol{\cX}_{0,K}) \\ 
\end{CD}
\end{equation}
commutative. Moreover, any object $(E,\nabla)$ in 
$\MIC(\cX_K,\ol{\cX}_K)$ is a direct summand of $f^*f_*(E,\nabla)$. 
Note that we have similar functors also in the convergent case 
(because the convergent case is a special case of overconvergent case): 
We have the the push-forward functors 
\begin{align*}
& f_*: \Isoc(X) \lra \Isoc(X_0), \\ 
& f_*: \MIC(\cX_K) \lra \MIC(\cX_{0,K})
\end{align*}
which makes the diagram 
\begin{equation}\label{push2}
\begin{CD}
\Isoc(X) @>{\Phi_{\cX}}>> 
\MIC(\cX_K) \\ 
@V{f_*}VV @V{f_*}VV \\ 
\Isoc(X_0) 
@>{\Phi_{\cX_0}}>> 
\MIC(\cX_{0,K}) \\ 
\end{CD}
\end{equation}
commutative. Moreover, the diagrams \eqref{push1}, \eqref{push2} are 
compatible via restriction functors $\alpha_{\cX_0}^*$ because of the 
functoriality of $\Phi_{(?,?)}$ and $f_*$. \par 
Now let us assume that the theorem is true for $(\cX_0,\ol{\cX}_0)$ 
(and $\cU = \cX_0$) and let us take an object $(E,\nabla)$ in 
$\MIC(\cX_K,\ol{\cX}_K)$ such that $\alpha_{\cX}^*(E,\nabla)$ is 
convergent. Then $f_*\alpha_{\cX}^*(E,\nabla) = \alpha_{\cX_0}^*f_*(E,\nabla)$ 
is convergent by the diagram \eqref{push2} and since we have assumed the 
theorem  for $(\cX_0,\ol{\cX}_0)$, this implies the overconvergence of 
$f_*(E,\nabla)$. Then $f^*f_*(E,\nabla)$ is overconvergent and so 
$(E,\nabla)$ is also overconvergent because it is a direct summand. 
(See the characterization of overconvergence in \cite[2.1]{co}.) 
Therefore we have shown that it suffices to prove the theorem for 
for $(\cX_0,\ol{\cX}_0)$ 
(and $\cU = \cX_0$). \par 
\ul{Step 4}: Here we prove the theorem for $(\cX_0,\ol{\cX}_0)$ above, by 
using \cite[2.7]{co}. Let $(E,\nabla)$ be an object in 
$\MIC(\cX_{0,K},\ol{\cX}_{0,K})$. In this case, $(E,\nabla)$ is defined on 
$A^n_K[\lambda,1] \times A^m_K[0,1]$ for some $\lambda \in [0,1) 
\cap \Gamma^*$. So we have the notion of the intrinsic generic radius 
of convergence $IR(E,\rho)$ for $\rho \in [\lambda,1]^n \times [0,1]^m$
which is defined by Kedlaya-Xiao \cite{kedlayaxiao}. By \cite[2.7]{co}, 
$(E,\nabla)$ is overconvergent if and only if 
$IR(E,\1) = 1$, where $\1 := (1,...,1)$. On the other hand, 
by applying \cite[2.1]{co} 
(see also \cite[2.5.6--8]{kedlayaI}, \cite[2.2.13]{berthelotrig}) to 
$(\cX_{0,K},\cX_{0,K})$, we see the following: 
When we fix a set of generators $(e_{\alpha})_{\alpha}$ of 
$\Gamma(\cX_{0,K},E)$, 
$\alpha_{\cX_0}^*(E,\nabla)$ is convergent if and only if, 
for each $\eta \in (0,1)\cap\Gamma^*$ and any $\alpha$, 
the multi-sequence 
$$\left\{ \left\| \dfrac{1}{i_1!\cdots i_{n+m}!}\pa^{i_1}_1\cdots
\pa^{i_{n+m}}_{n+m} 
(e_{\alpha}) \right\| \eta^{i_1+\cdots +i_{n+m}} \right\}_{i_1,...,i_{n+m}}$$ 
tends to zero as $i_1,...,i_{n+m} \to\infty$, where 
$\pa_j := \dfrac{\pa}{\pa t_j}$ and 
$\| \cdot \|$ denotes 
any $p$-adic Banach norm on $\Gamma(\cX_{0,K},E)$ induced by the 
affinoid norm on $\Gamma(\cX_{0,K},\cO)$. 
Then, by exactly the same argument as the proof of \cite[2.7]{co}, we
 see that the above condition is equivalent to the condition that 
$IR(E,\1) = \min_i\{p^{-1/(p-1)}|\pa_i|^{-1}_{E,\1,\sp}\} > \eta$ 
for each $\eta \in (0,1)\cap\Gamma^*$ (where $|\pa_i|_{E,\1,\sp}$ is 
the spectral norm of $\pa_i$ on $E$ `at the generic point of radius $\1$'), 
and it is nothing but the condition $IR(E,\1)=1$. Hence 
$(E,\nabla)$ is overconvergent if and only if 
$\alpha_{\cX_0}^*(E,\nabla)$ is convergent and so we are done. 
\end{proof}

We have the following corollary for convergent $F$-isocrystals. 

\begin{corollary}\label{cor}
Let $(X,\ol{X})$ be a smooth pair, let $U$ be an dense open subscheme of $X$ 
and let $\cE$ be an object in $\FIsoc(U)$ satisfying the following 
condition $(*)$: \\ 
\quad \\
$(*)$: \,\,\, There exists a Zariski open covering 
$\ol{X} = \bigcup_{\alpha\in\Delta}\ol{X}_{\alpha}$, 
lifts $(X \cap \ol{X}_{\alpha}, \ol{X}_{\alpha} \hra (\cX_{\alpha}, 
\ol{\cX}_{\alpha})$ and endomorphisms 
$F: \ol{\cX}_{\alpha} \lra 
\ol{\cX}_{\alpha}$ lifting the $q$-th power Frobenius on $\ol{X}_{\alpha}$ 
$($for $\alpha \in \Delta)$ satisfying the following$:$ If we denote the 
open formal subscheme of $\ol{\cX}_{\alpha}$ with special fiber 
$U \cap \ol{X}_{\alpha}$ by 
$\cU_{\alpha}$, the image of $\cE$ by the 
composite 
\begin{equation}\label{func1}
\FIsoc(U) \lra \FIsoc(U\cap \ol{X}_{\alpha}) 
\os{\Phi_{\cU_{\alpha}}}{\lra} \FMIC(\cU_{\alpha})
\end{equation}
is in the essential image of the functor 
\begin{equation}\label{func2}
\FMIC(\cX_{\alpha,K},\ol{\cX}_{\alpha,K}) \lra 
\FMIC(\cU_{\alpha}).
\end{equation}
\quad \\
Then $\cE$ is in the essential image of the functor 
$\FIsocd(X,\ol{X}) \lra \FIsoc(U).$
\end{corollary}

\begin{proof}
Assume that $\cE$ is sent as $\cE \mapsto \cE_{\alpha} \mapsto 
E_{\alpha}$ by the functors in \eqref{func1} 
and let $\ti{E}_{\alpha}$ be an object in 
$\FMIC(\cX_{\alpha,K},\ol{\cX}_{\alpha,K})$ which is sent to 
$E_{\alpha}$ by the functor \eqref{func2}. Then, by Theorem \ref{coc}, 
$\ti{E}_{\alpha}$ is overconvergent, that is, there exists an object 
$\ti{\cE}_{\alpha} \in \FIsocd(X \cap \ol{X}_{\alpha}, \ol{X}_{\alpha})$ 
which is sent to $\ti{E}_{\alpha}$ by 
$\Phi_{(\cX_{\alpha}, \ol{\cX}_{\alpha})}$. Now let us consider the 
following commutative diagram: 
\begin{equation*}
\begin{CD}
\FIsocd(X \cap \ol{X}_{\alpha}, \ol{X}_{\alpha}) 
@>{\Phi_{(\cX_{\alpha}, \ol{\cX}_{\alpha})}}>>
\FMIC(\cX_{\alpha,K},\ol{\cX}_{\alpha,K}) \\ 
@VVV @VVV \\ 
\FIsoc(U\cap \ol{X}_{\alpha}) 
@>{\Phi_{\cU_{\alpha}}}>> \FMIC(\cU_{\alpha}). 
\end{CD}
\end{equation*}
From the commutativity of the diagram, we see that both 
$\cE_{\alpha}$ and 
the restriction of $\ti{\cE}_{\alpha}$ to $\FIsoc(U\cap \ol{X}_{\alpha})$ 
is sent to $E_{\alpha}$ by $\Phi_{\cU_{\alpha}}$. Since 
$\Phi_{\cU_{\alpha}}$ is fully faithful, we can conclude that 
they are isomorphic. \par 
Now let us glue the $\ti{\cE}_{\alpha}$'s: Since 
$\cE_{\alpha}$'s $(\alpha \in \Delta)$ are 
the restriction of $\cE$, they form a descent data on 
$\FIsoc(U\cap \ol{X}_{\alpha}\cap  \ol{X}_{\alpha'})$ and 
$\FIsoc(U\cap \ol{X}_{\alpha}\cap  \ol{X}_{\alpha'} \cap \ol{X}_{\alpha''})$ 
$(\alpha,\alpha',\alpha'' \in \Delta)$. Then, since the restriction functor 
$$ \FIsocd(X \cap \ol{X}_{\alpha}, \ol{X}_{\alpha}) 
\lra \FIsoc(U\cap \ol{X}_{\alpha})$$ 
(and the functor we obtain by replacing $\ol{X}_{\alpha}$ by 
$\ol{X}_{\alpha} \cap \ol{X}_{\alpha'}$ or 
$\ol{X}_{\alpha} \cap \ol{X}_{\alpha'} \cap \ol{X}_{\alpha''}$) is 
fully faithful by \cite[4.1.1]{tsuzuki} and \cite[4.2.1]{kedlayaII} 
(see also \cite{kedlayaff}), 
we see that $\ti{\cE}_{\alpha}$'s also form a descent data on 
$\FIsocd(X \cap \ol{X}_{\alpha}\cap \ol{X}_{\alpha'}, 
\ol{X}_{\alpha} \cap \ol{X}_{\alpha'})$ and 
$\FIsocd(X \cap \ol{X}_{\alpha}\cap \ol{X}_{\alpha'} \cap \ol{X}_{\alpha''}, 
\ol{X}_{\alpha} \cap \ol{X}_{\alpha'} \cap \ol{X}_{\alpha''})$ 
 $(\alpha,\alpha',\alpha'' \in \Delta)$. Hence $\ti{\cE}_{\alpha}$'s 
glue to an object $\ti{\cE}$ in 
$\FIsocd(X,\ol{X})$ whose restriction to $\FIsoc(U)$ is 
$\cE$. So we are done. 
\end{proof}

\section{Purity for overconvergence}

In this section, we will prove the following theorem, 
which is the main result in this paper. 

\begin{theorem}\label{main}
Let $(X,\ol{X})$ be a smooth pair and 
let $\ol{Z} \subseteq Z \subseteq \ol{X}$ be closed 
subschemes of codimension at least $2$. Then the restriction 
functor 
\begin{equation}\label{mf}
\FIsocd(X,\ol{X}) \lra \FIsocd(X \setminus Z, 
\ol{X} \setminus \ol{Z})
\end{equation}
is an equivalence of categories. 
\end{theorem}

Note that, in the case where $\ol{Z}$ is empty, 
this is reduced to the purity theorem of Kedlaya \cite[5.3.3]{kedlayaI}. \par 
Before the proof, note that the functor \eqref{mf} is fully faithful 
because we have the fully faithful functors 
$$ \FIsocd(X,\ol{X}) \lra \FIsocd(X \setminus Z,\ol{X}) \lra 
\FIsoc(X \setminus Z), $$ 
$$ \FIsocd(X \setminus Z, \ol{X} \setminus \ol{Z}) \lra 
\FIsoc(X \setminus Z). $$
(The full faithfulness of the first functor in the first line is proven 
in \cite[4.1.1]{tsuzuki} and the full faithfulness for other functors 
are proven in \cite[4.2.1]{kedlayaII}.) So it suffices to prove that the 
functor \eqref{mf} is essentially surjective. 

\begin{proof}
The proof is divided into several steps. \par 
\ul{Step 1}: First note that it suffices to prove 
the theorem in the case $Z=\ol{Z}$. Indeed, we can factorize the 
functor \eqref{mf} as 
$$ \FIsocd(X,\ol{X}) \lra \FIsocd(X \setminus \ol{Z}, \ol{X} \setminus \ol{Z})
\lra \FIsocd(X \setminus Z, \ol{X} \setminus \ol{Z}) $$ 
and the second functor is an equivalence of categories by 
\cite[5.3.3]{kedlayaI}. 
So it suffices to prove the equivalence of the first functor. 
In the following, we always assume that $Z = \ol{Z}$. \par 
\ul{Step 2}: Next, note that it suffices to prove the theorem in the 
following cases: 
\begin{enumerate}
\item[(a)] The case $Z \subseteq \ol{X} \setminus X$. 
\item[(b)] The case $Z \subseteq X$. 
\end{enumerate}
Indeed, if we put $Z' := Z \cap (\ol{X} \setminus X)$, 
we can factorize the functor \eqref{mf} (with $Z=\ol{Z}$) as 
$$ \FIsocd(X,\ol{X}) \lra \FIsocd(X, \ol{X} \setminus Z')
\lra \FIsocd(X \setminus Z, \ol{X} \setminus Z) $$ 
and the equivalence of the first (resp. the second) functor 
follows from the theorem in the case (a) (resp. (b)). \par 
\ul{Step 3}: Let us prove that we may assume that $Z$ is smooth. 
For $n \in \N$, let us denote the $q^n$-th power map 
$\Spec k \lra \Spec k$ simply by $q^n$ and for a $k$-scheme $S$, 
let us put $S^{(n)} := S \times_{\Spec k,q^n} \Spec k$. Then the 
commutative diagram
\begin{equation*}
\begin{CD}
(X,\ol{X}) @>{F_{\rm rel}}>> (X^{(n)},\ol{X}^{(n)}) @>>> (X,\ol{X}) \\ 
@VVV @VVV @VVV \\ 
\Spec k @= \Spec k @>{q^n}>> \Spec k \\ 
@V{\cap}VV @V{\cap}VV @V{\cap}VV \\ 
\Spf O_K @= \Spf O_K @>{{\sigma^*}^n}>> \Spf O_K
\end{CD}
\end{equation*}
(where the upper left square is Cartesian and 
$F_{\rm rel}$ is the relative Frobenius morphism associated to 
the $q^n$-th power maps) induces the functors 
\begin{equation}\label{ab}
\FIsocd(X,\ol{X}) \os{\alpha}{\lra} \FIsocd(X^{(n)},\ol{X}^{(n)}) 
\os{\beta}{\lra} \FIsocd(X,\ol{X}).
\end{equation}
Then, for $(\cE, \Psi) \in \FIsocd(X,\ol{X})$, we have 
$$ 
\beta \circ \alpha (\cE, \Psi) = ({F^*}^n\cE, {F^*}^n\Psi) 
\os{\cong}{\lra} (\cE, \Psi), $$
where the last isomorphism is the composite 
$$ 
({F^*}^n\cE, {F^*}^n\Psi) \os{{F^*}^{n-1}\Psi}{\lra}
({F^*}^{n-1}\cE, {F^*}^n\Psi) \os{{F^*}^{n-1}\Psi}{\lra} \cdots 
\os{\Psi}{\lra} (\cE, \Psi). $$
Hence $\beta$ is essentially surjective. By the same argument, we see
 that the functors 
\begin{align*}
& \FIsocd(X^{(n)}, \ol{X}^{(n)} \setminus Z^{(n)}) \lra 
\FIsocd(X,\ol{X} \setminus Z), \quad \text{(in the case (a))} \\ 
& \FIsocd(X^{(n)} \setminus Z^{(n)}, \ol{X}^{(n)} \setminus Z^{(n)}) \lra 
\FIsocd(X \setminus Z,\ol{X} \setminus Z) \quad \text{(in the case (b))}  
\end{align*}
are also essentially surjective. So, to prove the theorem for the pair 
$(X,\ol{X})$ and the closed subscheme $Z$ in the case (a) or (b), 
it suffices to prove the theorem for the pair 
$(X^{(n)},\ol{X}^{(n)})$ and the closed subscheme $Z^{(n)}$ for some $n$, 
and then it suffices to prove the theorem for the pair 
$(X^{(n)},\ol{X}^{(n)})$ and the closed subscheme 
$Z^{(n)}_{\rm red}$($=$ the reduced closed subscheme of $Z^{(n)}$ with same 
underlying topological space). \par 
Now let us take $n \in \N$ in order that $Z^{(n)}_{\rm red}$ is 
generically smooth. 
Let $Z_0 \subseteq Z^{(n)}_{\rm red}$ be a dense open subscheme such that 
$Z_0$ is smooth and put 
$Z_1 := Z^{(n)}_{\rm red} \setminus Z_0$. Then in the case (a), we
 can factorize the functor 
$$ \FIsocd(X^{(n)}, \ol{X}^{(n)}) \lra 
\FIsocd(X^{(n)}, \ol{X}^{(n)} \setminus Z^{(n)}_{\rm red}) 
$$ 
(the functor \eqref{mf} for $(X^{(n)}, \ol{X}^{(n)})$ and 
$Z^{(n)}_{\rm red}$) as 
\begin{equation}\label{fac_sm}
\FIsocd(X^{(n)},\ol{X}^{(n)}) \ra \FIsocd(X^{(n)}, \ol{X}^{(n)} 
\setminus Z_1) \ra \FIsocd(X^{(n)}, \ol{X}^{(n)} 
\setminus Z^{(n)}_{\rm red})
\end{equation} 
and $(\ol{X}^{(n)} \setminus Z_1) \setminus 
(\ol{X}^{(n)} \setminus Z^{(n)}_{\rm red}) = Z_0$ is smooth. 
So, if we assume the theorem in the case $Z$ is smooth, the second functor 
in \eqref{fac_sm} 
is an equivalence. Moreover, the dimension of $\ol{X}^{(n)}$ is equal to the 
dimension of $\ol{X}$ and  
the codimension of $Z_1$ in $\ol{X}^{(n)}$ 
is strictly greater than that of $Z$ in $\ol{X}$. So we can prove 
the equivalence of the first functor in \eqref{fac_sm} 
by descending induction on the codimension 
of $Z$ in $\ol{X}$. We can prove the case (b) in the same way. \par 
\ul{Step 4}: Note that we can work Zariski locally on $\ol{X}$: 
Indeed, the categories $\FIsocd(X,\ol{X}), \FIsocd(X \setminus Z, 
\ol{X} \setminus Z)$ satisfy the descent property with respect to 
the Zariski topology on $\ol{X}$ and we already know the full faithfulness 
of the functor \eqref{mf} (or its analogue when we shrink $\ol{X}$), 
the descent data concerning 
$\FIsocd(X \setminus Z, \ol{X} \setminus Z)$ lifts to the descent data 
concerning $\FIsocd(X,\ol{X})$. So it suffices to prove the theorem 
Zariski locally on $\ol{X}$. \par 
\ul{Step 5}: Here we prove that it suffices to prove the theorem in 
the case (a) with $Z$ smooth. 
To prove this, it suffices to reduce the case (b) with $Z$ smooth 
to the case (a) with $Z$ smooth. 
So assume that we are in the case (b) with $Z$ smooth. 
Then, since $Z$ does not meet 
$\ol{X} \setminus X$ and we can work Zariski locally on $\ol{X}$, we may 
assume that either $Z$ or $\ol{X} \setminus X$ is empty. In the case 
$Z$ is empty, the theorem is obvious. In the case 
$\ol{X} \setminus X$ is empty, it suffices to prove the equivalence of 
the functor $\FIsocd(X,X) \lra \FIsocd(X \setminus Z, X \setminus Z).$ 
This functor is facorized as 
\begin{equation}\label{hojyo1}
\FIsocd(X,X) \lra \FIsoc(X \setminus Z, X) \lra 
\FIsocd(X \setminus Z, X \setminus Z) 
\end{equation}
and the first functor is an equivalence by \cite[5.3.3]{kedlayaI}. 
So it suffices 
to prove that the second functor is an equivalence. Since we can consider 
Zariksi locally and since $Z$ is smooth, 
we may assume that there exists 
a smooth divisor $D \subseteq X$ containing $Z$. 
Then, by \cite[5.3.7]{kedlayaI} we have the canonical equivalence 
of categories
\begin{equation*}
\FIsocd(X \setminus Z, X) \os{=}{\lra} 
\FIsocd(X \setminus Z, X \setminus Z) 
\times_{\FIsocd(X \setminus D, X \setminus Z)}
\FIsocd(X \setminus D, X)
\end{equation*}
via which the second functor in \eqref{hojyo1} is regarded as 
the first projection 
$$ \FIsocd(X \setminus Z, X \setminus Z) 
\times_{\FIsocd(X \setminus D, X \setminus Z)}
\FIsocd(X \setminus D, X) \lra \FIsocd(X \setminus Z, X \setminus Z). $$
So it suffices to prove the equivalence of the functor 
$\FIsocd(X \setminus D, X) \lra \FIsocd(X \setminus D, X \setminus Z)$ 
to prove the theorem in this case, and it is nothing but the theorem 
in the case (a). it suffices to prove the theorem in 
the case (a) with $Z$ smooth. \par 
\ul{Step 6}: 
Let us put $Y := \ol{X} \setminus X$ (which is a simple normal 
crossing divisor in $\ol{X}$) 
and let $Y = \bigcup_{i=1}^r Y_i$ be the decomposition 
of $Y$ into the irreducible components (so $r$ denotes the number of 
irreducible components of $Y$). 
For a subset $I \subseteq [1,r]$, 
we put $Y_I := \bigcap_{i\in I} Y_i$ and for $s \in \N$, we put 
$Y^{(s)} := \bigcup_{|I|=s} Y_I$. In this step, we show that it suffices 
to prove the theorem in the following cases: 
\begin{enumerate}
\item[(a-$1$)]: The case (a) with 
$r=1$ and $Z$ smooth. 
\item[(a-$s$)]\,($s \geq 2$): The case (a) with 
$r=s$, $Y_{[1,s]} \not= \emptyset$ and $Z = Y^{(2)}$. 
\end{enumerate}
(Note that, in the cases (a-$s$)\,($s \geq 2$), $Z$ is no more smooth.) \par 
Let us assume that the theorem is true in the case 
(a-$s$) \,($s \geq 1$) and let $(X,\ol{X}), Z\subseteq \ol{X}$ be as in 
case (a) with $Z$ smooth. First, let us note that, to prove the theorem 
for this $(X,\ol{X})$ and $Z \subseteq \ol{X}$, 
it suffices to prove the equivalence of the composite functor 
\begin{equation}\label{hojyo3}
\FIsocd(X,\ol{X}) \lra \FIsocd(X,\ol{X} \setminus Z) \lra 
\FIsocd(X,\ol{X} \setminus (Y^{(2)} \cup Z)), 
\end{equation}
because they are fully faithful. 
We can factor the functor 
\eqref{hojyo3} as follows: 
$$ 
\FIsocd(X,\ol{X}) 
\lra \FIsocd(X,\ol{X} \setminus Y^{(2)}) 
\lra  \FIsocd(X,\ol{X} \setminus (Y^{(2)} \cup Z)). 
$$
So, to prove the theorem, it suffices to prove the equivalence of the functors 
\begin{align}
&\FIsocd(X,\ol{X} \setminus Y^{(2)}) \lra 
\FIsocd(X,\ol{X} \setminus (Y^{(2)} \cup Z)), \label{hojyo4}\\ 
&\FIsocd(X,\ol{X}) \lra 
\FIsocd(X,\ol{X} \setminus Y^{(2)}) \label{hojyo5}. 
\end{align}
The equivalence of the functor \eqref{hojyo4} is nothing but the theorem 
for the pair $(X,\ol{X} \setminus Y^{(2)})$ and the closed subscheme 
$Z \setminus Y^{(2)} \subseteq \ol{X} \setminus Y^{(2)}$. Note that 
$Y \setminus Y^{(2)} = (\ol{X} \setminus Y^{(2)}) \setminus X$ is a 
smooth divisor and $Z \setminus Y^{(2)}$ is smooth. Hence we are locally in 
the situation in the case (a) with $r=1$, $Z$ smooth, that is, 
the situation in the case (a-$1$). On the other hand, we can 
reduce the equivalence of the functor \eqref{hojyo5} to 
(a-$s$) \,($s \geq 2$) in the following way: For any closed point 
$x \in \ol{X}$, let $Y_{\ni x} := \bigcup_{i \in I_x} Y_i$ be the 
union of irreducible components of $Y$ containing $x$. Then we can take 
an open subscheme $\ol{X}_x$ of $\ol{X}$ containing $x$ such that 
$Y \cap \ol{X}_x = Y_{\ni x} \cap \ol{X}_x$. Since we can consider 
Zariski locally, the equivalence of the functor \eqref{hojyo5} is reduced 
to the theorem for the pair $(X \cap \ol{X}_x, \ol{X}_x)$ and 
the closed subscheme $(Y \cap \ol{X}_x)^{(2)} = Y_{\ni x}^{(2)} \cap 
\ol{X}_x$. Let us put $s := |I_x|$. Then, when $s \geq 2$, 
it is in the situation (a-$s$). When $s \leq 1$, it is trivially true 
since $Y_{\ni x}^{(2)} = \emptyset$ in this case. 
Therefore, we are reduced to the case
 (a-$s$) \,($s \geq 1$). \par 
\ul{Step 7}: In this step, we show that it suffices 
to prove the theorem in the case (a-$s$) with $s \geq 2$. To do this, 
it suffices to reduce the proof of the theorem in the case 
(a-$1$) to the case (a-$2$). So assume that we are in the situation 
(a-$1$). Then, since we may consider locally, we may assume that 
there exists a smooth divisor $Y' \subseteq \ol{X}$ which meets 
$Y$ transversally such that $Y' \cap Y$ contains $Z$. In this case, 
it suffices to prove the equivalence of the composite functor 
$$ \FIsoc(X,\ol{X}) \lra \FIsoc(X,\ol{X} \setminus Z) \lra 
\FIsoc(X,\ol{X} \setminus (Y \cap Y')). $$
Then, by \cite[5.3.7]{kedlayaI}, 
we can rewrite the above composite as follows: 
\begin{align*}
\FIsocd(X,\ol{X}) & \os{=}{\lra} 
\FIsocd(X \setminus Y', \ol{X}) 
\times_{\FIsocd(X \setminus Y', X)} 
\FIsocd(X,X) \\ 
& \lra 
\FIsocd(X \setminus Y', \ol{X} \setminus (Y \cap Y'))  
\times_{\FIsocd(X \setminus Y', X)} 
\FIsocd(X,X)  \\ & \os{=}{\lla} 
\FIsocd(X,\ol{X} \setminus (Y \cap Y')). 
\end{align*}
So the equivalence we need is reduced to the equivalence of the functor 
$$ \FIsocd(X \setminus Y', \ol{X}) \lra 
\FIsocd(X \setminus Y', \ol{X} \setminus (Y \cap Y')) $$ 
and this is the theorem for the pair 
$(X \setminus Y', \ol{X})$ and the closed subscheme 
$Y \cap Y'\subseteq \ol{X}$. Since this is contained in 
the case (a-$2$), we are done. \par 
\ul{Step 8}: Assume now that we are in the situation (a-$s$) with 
$s \geq 2$. 
Let us take a closed point $x \in \ol{X}$. 
Let $Y_{\ni x} = \bigcup_{i=1}^{s'} Y_i$ be the union of irreducible 
components of $Y$ containing $x$ ($s' \leq s$) and 
let $Y_{\not\ni x}$ be the union of irreducible 
components of $Y$ not containing $x$. By applying 
\cite[Theorem 2]{kedlayamore} to $\ol{X} \setminus 
Y_{\not\ni x}$ and the simple 
normal crossing divisor $Y \setminus Y_{\not\ni x} = 
Y_{\ni x} \setminus Y_{\not\ni x}$ 
on it, we see that there exists an open subscheme $\ol{X}_x$ in 
$\ol{X} - Y_{\not\ni x}$ containing $x$ and a finite etale morphism 
$f_0: \ol{X}_x \lra \Af^{d}_k$ for some $d \geq s'$ such that, for 
$1 \leq i \leq s'$, $f_0(Y_i \cap \ol{X}_x)$ is contained in 
the $i$-th coordinate hyperplane $H_i$ of $\Af^{d}_k$. Then 
$Y_i \cap \ol{X}_x \subseteq f_0^{-1}(H_i)$ 
is an open and closed immersion, and so 
$Y \cap \ol{X}_x = Y_{\ni x} \cap \ol{X}_x$ is a simple normal 
crossing subdivisor of $\bigcup_{i=1}^n f_0^{-1}(H_i)$. \par 
Since we can consider Zariski locally, it suffices to prove the theorem 
for the pair $(X \cap \ol{X}_x, \ol{X}_x)$ and the closed subscheme 
$Z \cap \ol{X}_x = Y^{(2)} \cap \ol{X}_x 
= Y_{\ni x}^{(2)} \cap \ol{X}_x$. 
It is contained in the case (a-$s'$) when $s \geq 2$. When $s' \leq 1$, 
the theorem in this case is trivially true since $Y^{(2)}_{\ni x} = 
\emptyset$. 
Summing up the argument here, we see the 
following: It suffices to prove the theorem in the situation 
(a-$s$) with $s \geq 2$ which admits a finite etale morphism 
$f: \ol{X} \lra \Af^d_k$ (for some $d \geq s$) such that 
$Y = \ol{X} \setminus X$ is a simple normal crossing subdivisor 
of $\ti{Y} := \bigcup_{i=1}^s f^{-1}(H_i)$, where $H_i$ denotes the 
$i$-th coordinate hyperplane of $\Af^d_k$. \par 
\ul{Step 9}: Let the notation be as above and 
let us put $X' := X \setminus \ti{Y}, 
\ti{Z} := \ti{Y}^{(2)}$. Consider the following commutative 
diagram: 
\begin{equation}\label{hojyo10}
\begin{CD}
\FIsocd(X,\ol{X}) @>>> \FIsocd(X',\ol{X}) @>>> 
\FIsocd(X',\ol{X} \setminus \ti{Z}) \\ 
@VVV @VVV @VVV \\ 
\FIsocd(X,X) @>>> \FIsocd(X',X) @>>> \FIsocd(X',X \setminus \ti{Z}). 
\end{CD}
\end{equation}
Now assume that the theorem is true for the pair $(X',\ol{X})$ and the 
closed subscheme $\ti{Z} \subseteq \ol{X}$, and let us take an object 
$\cE$ in $\FIsocd(X,\ol{X} \setminus Z)$. Then the restriction of 
$\cE$ to $\FIsocd(X',\ol{X} \setminus \ti{Z})$ extends to an object 
$\cF$ in $\FIsocd(X',\ol{X})$. The restriction of $\cF$ to 
$\FIsocd(X',X)$ is canonically isomorphic to the restriction of 
$\cE$ because so do they in the category 
$\FIsocd(X',X \setminus \ti{Z})$ (note that all the functors in 
\eqref{hojyo10} are fully faithful). Now let us note that the left square 
is Cartesian in the sense that the induced functor 
\begin{equation}\label{hojyo11}
\FIsocd(X,\ol{X}) \lra \FIsocd(X',\ol{X}) \times_{\FIsocd(X',X)} 
\FIsocd(X,X)
\end{equation}
is an equivalence of categories (\cite[5.3.7]{kedlayaI}). 
Hence the object 
$(\cF, \text{the restriction of $\cE$})$ in the target of \eqref{hojyo11} 
lifts to an object $\ti{\cE}$ in $\FIsocd(X,\ol{X})$ and it gives the 
lift of $\cE$. So we see that it suffices to prove the theorem for 
the pair $(X',\ol{X})$ and the closed subscheme $\ti{Z} \subseteq \ol{X}$: 
That is, to prove the theorem, we may assume that 
there exists a finite etale morphism $f: \ol{X} \lra \Af^d_k$ and some 
$s \leq d$ such that $Y := \ol{X} \setminus X = \bigcup_{i=1}^s f^{-1}(H_i)$ 
and that $Z = Y^{(2)}$. (Note that we do not assume the condition 
(a-$s$) any more.) \par 
\ul{Step 10}: Let the notation be as above that let us put 
$\ol{X}_0 := \Af^d_k, X_0 := \Af^d_k \setminus \bigcup_{i=1}^s H_i, 
Y_0 := \bigcup_{i=1}^s H_i$, $Z_0 := Y_0^{(2)}$. Then, by 
\cite[5.1]{tsuzuki}, we have the puch-out functors $f_*$ as in the 
following commutative diagram: 
\begin{equation}\label{hojyo12}
\begin{CD}
\FIsocd(X,\ol{X}) @>>> \FIsocd(X,\ol{X} \setminus Z) \\ 
@V{f_*}VV @V{f_*}VV \\ 
\FIsocd(X_0,\ol{X}_0) @>>> \FIsocd(X_0, \ol{X}_0 \setminus Z_0). 
\end{CD}
\end{equation}
Moreover, it is known that, for any $\cE \in \FIsocd(X,\ol{X} \setminus Z)$, 
there exist morphisms $\cE \os{\alpha}{\hra} 
f^*f_*\cE \os{\beta}{\twoheadrightarrow} \cE$ which 
makes $\cE$ a direct summand of $f^*f_*\cE$. (See \cite[2.6.8]{kedlayaI}.) 
Now let us assume that the theorem is true for the pair 
$(X_0,\ol{X}_0)$ and the closed subscheme $Z_0 \subseteq \ol{X}_0$, and 
let us take an object $\cE$ in $\FIsocd(X,\ol{X} \setminus Z)$. 
Then there exists an object $\cF$ in $\FIsocd(X_0,\ol{X}_0)$ which 
restricts to $f_*\cE$ in $\FIsocd(X_0, \ol{X}_0 \setminus Z_0)$. 
Then $f^*\cF \in \FIsocd(X,\ol{X})$ 
restricts to $f^*f_*\cE$ in $\FIsocd(X,\ol{X} \setminus Z)$. 
Now let us note that $\cE$ is isomorphic to the image of the composite 
$\alpha \circ \beta: f^*f_*\cE \lra f^*f_*\cE$. Since the upper horizontal 
functor in \eqref{hojyo12} is fully faithful and exact, $\alpha \circ \beta$ 
lifts to an endomorphism $\gamma: f^*\cF \lra f^*\cF$ of $f^*\cF$ and 
${\rm Im} \gamma$ is an object in $\FIsocd(X,\ol{X})$ which restricts to 
$\cE$ in $\FIsocd(X,\ol{X} \setminus Z)$. So the theorem for $(X,\ol{X})$ and 
$Z \subseteq \ol{X}$ is true. Hence we have proved that 
it suffices to prove the theorem for the pair $(X_0,\ol{X}_0)$ and 
the closed subscheme $Z_0 \subseteq \ol{X}_0$. \par 
\ul{Step 11}: Let $(X_0,\ol{X}_0)$ and $Z_0 \subseteq \ol{X}_0$ be as above. 
In this step, we finish the proof of the theorem by giving 
a proof of the theorem for $(X_0,\ol{X}_0)$ and 
$Z_0 \subseteq \ol{X}_0$. \par 
Let us put $\ol{\cX}_0 := \fAf^d_{O_K}, \cX_0 := 
\fG^s_{m,O_K} \times \fAf^{d-s}_{O_K}$. Let $t_1,...,t_d$ be the coordinate 
function of $\ol{\cX}_0$ and let $F: \ol{\cX}_0 \lra \ol{\cX}_0$ be the 
morphism over $\sigma^*: \Spf O_K \lra \Spf O_K$ defined by 
$F^*(t_i) := t_i^{q} \,(1 \leq i \leq d)$. 
Then $(\cX_0,\ol{\cX}_0)$ is a formal smooth pair with special fiber 
$(X_0,\ol{X}_0)$ and $F$ defines an endomorphism on it which lifts the 
$q$-th power Frobenius on $\ol{X}_0$. Let $\ol{\cX}'_0$ 
be the open formal 
subscheme of $\ol{\cX}_0$ with special fiber $\ol{X}_0 \setminus Z_0$. 
Then we have the commutative diagram 
\begin{equation*}
\begin{CD}
\FIsocd(X_0, \ol{X}_0 \setminus Z_0) 
@>{\Phi_{(\cX_0,\ol{\cX}'_0)}}>> \FMIC(\cX_{0,K},\ol{\cX}'_{0,K}) \\ 
@VVV @VVV \\ 
\FIsoc(X_0) @>{\Phi_{\cX_0}}>> \FMIC(\cX_{0,K}). 
\end{CD}
\end{equation*}
Let $\cE$ be an object in $\FIsocd(X_0, \ol{X}_0 \setminus Z_0)$ and 
put $E := \Phi_{(\cX_0,\ol{\cX}'_0)}(\cE)$. To prove the theorem, 
it suffices to find an object in $\FIsocd(X_0, \ol{X}_0)$ whose 
restriction to $\FIsoc(X_0)$ is isomorphic to the restriction of $\cE$ to 
$\FIsoc(X_0)$, because the functor 
$\FIsocd(X_0, \ol{X}_0 \setminus Z_0) \lra \FIsoc(X_0)$ is 
fully faithful. Then, by Corollary \ref{cor}, it suffices to prove that the 
restriction of $E$ to $\FMIC(\cX_{0,K})$ is extendable to an object in 
$\FMIC(\cX_{0,K},\ol{\cX}_{0,K})$. So we see that 
it suffices to prove the following claim: \\
\quad \\
{\bf claim 1}: Let $E$ be an object in $\FMIC(\cX_{0,K},\ol{\cX}'_{0,K})$. Then there exists an object $\ti{E}$ in $\FMIC(\cX_{0,K},\ol{\cX}_{0,K})$ such 
that their restrictions to $\FMIC(\cX_{0,K})$ are isomorphic. \\
\quad \\
For a subset $I$ of $[1,s]$, let 
$\ol{X}'_{0,I}$ be the open subscheme of 
$\ol{X}_0$ defined as 
$\ol{X}'_{0,I} := \ol{X}_0 \setminus 
\{\prod_{j \in [1,s] \setminus I} t_j = 0\}$ and 
let $\ol{\cX}'_{0,I}$ be the open formal subscheme of $\ol{\cX}_0$ whose 
special fiber is equal to $\ol{X}'_{0,I}$. 
Then 
\begin{align*}
\fU_{0,I,\lambda} 
& := \{x \in \ol{\cX}'_{0,I,K} \,|\,\forall j \in I, |t_j(x)| \geq \lambda\} \\ & = \{x \in \ol{\cX}_{0,K} \,|\,\forall j \in [1,s] \setminus I, 
|t_j(x)| = 1 \text{ and } \forall j \in I, |t_j(x)| \geq \lambda\} \\ 
& \hspace{8.5cm} (\lambda \in [0,1) \cap \Gamma^*)
\end{align*}
gives a cofinal system of strict neighborhoods of 
$\cX_{0,K}$ in $\ol{\cX}'_{0,I,K}$. Let us put 
$$ 
A_{I,K} := \varinjlim_{\lambda\to 1}\Gamma(\fU_{0,I,\lambda}, 
\cO_{\fU_{0,I,\lambda}}) 
= \Gamma(\ol{\cX}'_{0,I,K}, j^{\dagger}\cO_{\ol{\cX}'_{0,I,K}})
$$ 
(where $j^{\dagger}\cO_{\ol{\cX}'_{0,I,K}}$ denotes the sheaf of 
overconvergent sections for the formal smooth pair 
$(\cX_0,\ol{\cX}'_{0,I})$). Note that $A_{I,K}$ admits the canonical ring 
homomorphism $F^*: A_{I,K} \lra A_{I,K}$ induced by $F: \ol{\cX}_0 \lra 
\ol{\cX}_0$. We define the category $\FMIC(A_{I,K})$ as the category 
of pairs $((E,\nabla),\Psi)$, where $(E,\nabla)$ is a projective 
$A_{I,K}$-module of finite rank endowed with an integrable connection 
$\nabla: E \lra \bigoplus_{i=1}^d E dt_i$ and $\Psi$ is an 
isomorphism $A \otimes_{F^*,A} (E,\nabla) \os{=}{\lra} (E,\nabla)$ 
(where $A \otimes_{F^*,A} -$ means the scalar extension by $F^*$ as 
module endowed with integrable connection). 
Then we have the equivalence of categories 
$$ 
\FMIC(\cX_{0,K},\ol{\cX}'_{0,I,K}) \os{=}{\lra} 
\FMIC(j^{\dagger}\cO_{\ol{\cX}'_{0,I,K}}) \os{=}{\lra} 
\FMIC(A_{I,K}). 
$$
Since we have $\ol{\cX}'_{0,\emptyset} = \cX_0$ and 
$\ol{\cX}'_{0,[1,s]} = \ol{\cX}_0$, we have 
\begin{equation}\label{hojyo100}
\FMIC(\cX_{0,K}) \os{=}{\lra} 
\FMIC(A_{\emptyset,K}), \quad 
\FMIC(\cX_{0,K},\ol{\cX}_{0,K}) \os{=}{\lra} 
\FMIC(A_{[1,s],K})
\end{equation}
as particular cases. On the other hand, note that 
$\ol{\cX}'_0 \setminus Z_0 = \bigcup_{i=1}^s \ol{\cX}'_{0,\{i\}}$
and that 
$\ol{\cX}'_{0,\{i\}} \cap \ol{\cX}'_{0,\{i'\}} = \cX_0$ for any 
$1 \leq i,i' \leq s, i \not= i'$. So we have the equivalences 
\begin{align}
& \,\, \FMIC(\cX_{0,K},\ol{\cX}'_{0,K}) \label{hojyo101} \\ 
\os{=}{\lra} & 
\text{ fiber product of 
$\FMIC(\cX_{0,K},\ol{\cX}'_{0,\{i\},K}) \ra \FMIC(\cX_{0,K}) 
\,(i \in [1,s])$} \nonumber \\ 
\os{=}{\lra} & 
\text{ fiber product of 
$\FMIC(A_{\{i\},K}) \ra \FMIC(A_{\emptyset,K}) 
\,(i \in [1,s])$}. \nonumber 
\end{align}
For $I' \subseteq I \subseteq [1,s]$, let us denote `the 
scalar extension functor' 
$$\FMIC(A_{I,K}) \lra \FMIC(A_{I',K})$$
(which is induced by the canonical 
inclusion of the rings $A_{I,K} \hra A_{I',K}$) 
by $A_{I',K} \otimes_{A_{I,K}} -$. Then 
the claim 1 is equivalent to the following claim: \\
\quad \\
{\bf claim 2}: Let $((E_{\emptyset},\nabla_{\emptyset}), \Psi_{\emptyset})$ 
be an object in $\FMIC(A_{\emptyset,K})$ and for $1 \leq i \leq s$, 
let $((E_{\{i\}}, \nabla_{\{i\}}), \Psi_{\{i\}})$ 
be an object in $\FMIC(A_{\{i\},K})$ endowed with an isomorphism 
$$f_{\{i\}}: 
A_{\emptyset,K} \otimes_{A_{\{i\},K}} 
((E_{\{i\}}, \nabla_{\{i\}}), \Psi_{\{i\}}) \os{=}{\lra} 
((E_{\emptyset},\nabla_{\emptyset}), \Psi_{\emptyset}). $$ 
Then there exists an object 
$((E_{[1,s]},\nabla_{[1,s]}), \Psi_{[1,s]})$ 
in $\FMIC(A_{[1,s],K})$ endowed with an isomorphism 
$$ f_{[1,s]}: 
A_{\emptyset,K} \otimes_{A_{[1,s],K}} 
((E_{[1,s]},\nabla_{[1,s]}), \Psi_{[1,s]}) \os{=}{\lra} 
((E_{\emptyset},\nabla_{\emptyset}), \Psi_{\emptyset}). $$
\quad \\
Before the proof of claim 2, we prove preliminary facts on commutative 
algebra. \\
\quad \\ 
{\bf claim 3}: 
We have the following: 
\begin{enumerate}
\item The homomorphism 
$F^*: A_{I,K} \lra A_{I,K}$ is flat for any $I \subseteq [1,s]$. 
\item  
For $I, I' \subseteq [1,s]$, we have 
$A_{I,K} \cap A_{I',K} = A_{I \cup I',K}$. (Here the intersection is 
taken in $A_{\emptyset,K}$.) 
\item 
Let $I \subseteq [1,s]$. Then, 
For any projective $A_{\emptyset,K}$-module of finite rank $E_{\emptyset}$, 
there exists a projective $A_{I,K}$-module of finite rank 
$E_I$ endowed with an isomorphism 
$f_I: A_{\emptyset,K} \otimes_{A_{I,K}} E_I \os{=}{\lra} E_{\emptyset}$. 
If there are two such data $(E_I,f_I), (E'_I,f'_I)$, there exists an 
isomorphism $g: E_I \os{=}{\lra} E'_I$ such that the composite 
$$  E_{\emptyset} \os{f_I^{-1}}{\lra} A_{\emptyset,K} \otimes_{A_{I,K}} E_I 
\os{{\rm id} \otimes g}{\lra} 
 A_{\emptyset,K} \otimes_{A_{I,K}} E'_I \os{f'_I}{\lra} E_{\emptyset} $$ 
is the identity map. 
\end{enumerate}

We prove the claim 3. For $n := (n_1,...,n_d) \in \Z^d$ and 
$I \subseteq [1,d]$, we put $n_I := \sum_{i \in I}n_i$. 
Then, by definition, the ring $A_{I,K}$ has 
the following concrete description: 
$$ A_{I,K} := \{\sum_{n\in \Z^s \times \N^{d-s}} a_nt^n \,|\, 
a_n \in K, \exists \lambda \in [0,1), 
\forall J \subseteq I, |a_n| \lambda^{n_J} \to 0 \,(|n|\to\infty)\}. $$ 
(See \cite[3.1.7]{kedlayaI} for example.) 
Using this description, it is easy to see that $A_{I,K}$ is freely generated 
by $\{t^n\}_{n=(n_1,...,n_d), 0 \leq n_i \leq q-1}$ when it 
is regarded as an $A_{I,K}$-module via $F^*$. So we have (1). (2) 
follows from the above description and the inequality 
$$|a_n| \lambda^{n_J/2} \leq \max \{|a_n| \lambda^{n_{J\cap I}}, 
 |a_n| \lambda^{n_{J\setminus I}}\}$$
for any $J \subseteq I \cup I'$. To prove (3), we put 
$R_I := O_K\{t_i \,(1 \leq i \leq d), t_j^{-1}\,(j \in [1,s] \setminus I)\}$ 
and let $A_I$ be the weak completion of $R_I[t_j^{-1} \,(j \in I)]$ over  
$(R_I,pR_I)$ in the sense of Monsky-Washnitzer \cite{mw}. 
($A_I$ is a w.c.f.g. algebra over $(R_I,pR_I)$ by definition.) 
Then, using \cite[2.2, 2.3]{mw}, we see easily the equality 
$A_{I,K} = \Q \otimes_{\Z} A_I$ by direct calculation. Also, by 
\cite[Th\'eor\`eme 3]{etessemw}, $(\Spec A_I, \Spec A_I/pA_I)$ 
is a Henselian couple in the sense of \cite[18.5.5]{ega4-4}. 
Then we have the assertion (3) by 
\cite[Corollaire 1 du Th\'eor\`eme 3]{elkik}
 and \cite[Lemme in p.573]{elkik} 
(see also the proof of \cite[Proposition 3]{etessedescent}). \par 
Now we prove the claim 2, using claim 3. Note that 
it suffices to prove the following claim $(*)_j$ 
by induction on $j \,(1 \leq j \leq s)$: \\
\quad \\
$(*)_j$: \,With the assumption of the claim 2, there exists 
an object 
$((E_{[1,j]},\nabla_{[1,j]}), \Psi_{[1,j]})$ 
in $\FMIC(A_{[1,j],K})$ endowed with an isomorphism 
$$ f_{[1,j]}: 
A_{\emptyset,K} \otimes_{A_{[1,j],K}} 
((E_{[1,j]},\nabla_{[1,j]}), \Psi_{[1,j]}) \os{=}{\lra} 
((E_{\emptyset},\nabla_{\emptyset}), \Psi_{\emptyset}). $$
\quad \\
$(*)_1$ is true by assumption. Now we assume $(*)_{j-1}$ and prove 
$(*)_j$. (The argument in the following 
is inspired by that in \cite[I]{etessedescent}.) 
By (3) of claim 3, there exists a projective $A_{[1,j],K}$-module 
$E_{[1,j]}$ of finite rank endowed with an isomorphism 
$f_{[1,j]}: A_{\emptyset,K} \otimes_{A_{[1,j],K}} E_{[1,j]} \os{=}{\lra} 
E_{\emptyset}.$ Then, since both 
$A_{[1,j-1],K} \otimes_{A_{[1,j],K}} E_{[1,j]}$ and $E_{[1,j-1]}$ are 
isomorphic to $E_{\empty}$ after we apply 
$A_{\emptyset,K} \otimes_{A_{[1,j-1],K}} -$, they are canonically 
isomorphic. By the same reason, 
$A_{\{j\},K} \otimes_{A_{[1,j],K}} E_{[1,j]}$ and $E_{\{j\}}$ are 
canonically isomorphic. Since $E_{[1,j]}$ is projective, we can regard 
$E_{[1,j]}$ as a submodule of $E_{[1,j-1]}\cap E_{\{j\}}$ (the intersection 
is taken in $E_{\emptyset}$). Moreover, by 
claim 3 (2) and the projectivity of $E_{[1,j]}$, we see that 
$E_{[1,j]}$ is equal to $E_{[1,j-1]}\cap E_{\{j\}}$. 
Then, using claim 3 (1), we see that 
$A_{[1,j],K} \otimes_{F^*,A_{[1,j],K}} E_{[1,j]}$ is 
equal to 
$(A_{[1,j-1],K} \otimes_{F^*,A_{[1,j-1],K}} E_{[1,j-1]}) \cap 
(A_{\{j\},K} \otimes_{F^*,A_{\{j\},K}} E_{\{j\}})$. 
Then we can define 
$\nabla_{[1,j]}: E_{[1,j]} \lra \bigoplus_{i=1}^d E_{[1,j]} dt_i$ and 
$\Psi_{[1,j]}: A_{[1,j],K} \otimes_{F^*,A_{[1,j],K}} 
(E_{[1,j]},\nabla_{[1,j]}) \lra (E_{[1,j]},\nabla_{[1,j]})$ simply by 
\begin{align*}
& \nabla_{[1,j]} := \nabla_{[1,j-1]} |_{E_{[1,j]}} = 
\nabla_{\{j\}} |_{E_{[1,j]}}, \\ 
& \Psi_{[1,j]} := 
\Psi_{[1,j-1]} |_{A_{[1,j],K} \otimes_{F^*,A_{[1,j],K}} E_{[1,j]}} 
= 
\Psi_{\{j\}} |_{A_{[1,j],K} \otimes_{F^*,A_{[1,j],K}} E_{[1,j]}}. 
\end{align*}
Then it is easy to see that 
$((E_{[1,j]},\nabla_{[1,j]}), \Psi_{[1,j]})$ satisfies the required 
condition in $(*)_j$. So we have proved the claim 2 (hence the claim 1) 
and therefore 
the proof of the theorem is now finished. 
\end{proof}

\section{An application}

In this section, assume that $k$ is perfect. 
Let $X$ be a connected smooth scheme over $k$ and 
let $\pi_1(X)$ be the fundamental group of $X$ 
(we omit to write the base point). Denote by 
$\Rep_{K^{\sigma}}(\pi_1(X))$ the category of 
finite-dimensional continuous representations of $\pi_1(X)$ over 
$K^{\sigma}$, 
where $K^{\sigma}$ denotes the fixed field of $K$ by $\sigma$. 
In \cite{crew}, Crew proved that there exists the canonical equivalence 
\begin{equation}\label{creweq}
G: \Rep_{K^{\sigma}}(\pi(X)) \os{=}{\lra} \FIsoc(X)^{\circ}, 
\end{equation}
where $\FIsoc(X)^{\circ}$ denotes the category of unit-root convergent 
$F$-isocrystals on $X$, that is, the category of convergent 
$F$-isocrystals on $X$ satisfying certained condition called 
`unit-root condition'. (For the definition, see \cite{crew}.) \par 
Assume now that the above $X$ is enclosed to a pair $(X,\ol{X})$ and 
let $S$ be the set of all the discrete 
valuations of $k(X)$ centered on $\ol{X} \setminus X$. (As for the notion 
of the center of valuation, see \cite{kedlayaII} or \cite{vaquie}.) 
For $v \in S$, let us denote the 
inertia subgroup of $k(X)_v$ ($:=$ the completion of $k(X)$ with respect 
to $v$) by $I_v$. Then we have the natural homomorphism 
$I_v \lra \pi_1(X)$ which is well-defined up to conjugate. 
Then we define the subcategory of $\Rep_{K^{\sigma}}(\pi_1(X))$ 
with finite local monodromy along $S$ by 
$$ \Rep_{K^{\sigma}}^{S}(\pi_1(X)) := 
\{\rho \in \Rep_{K^{\sigma}}(\pi_1(X)) \,|\, 
\forall v \in S, \rho|_{I_v}\text{ has finite image}\}. $$ 
Tsuzuki has shown in \cite{tsuzukicurve} that, in the case where $X$ is 
a curve, the functor $G$ 
of Crew restricts to the equivalence of categories 
$$ \Rep_{K^{\sigma}}^{S}(\pi(X)) \os{=}{\lra} \FIsocd(X,\ol{X})^{\circ}, $$
where $\FIsocd(X,\ol{X})^{\circ}$ denotes the category of unit-root 
overconvergent $F$-isocrystals on $(X,\ol{X})$, which is the category 
$\FIsocd(X,\ol{X}) \cap \FIsoc(X)^{\circ}$. 
In higher dimensional case, Kedlaya \cite[2.3.7, 2.3.9]{kedlayaswanII}
proved the following result, based on a result of Tsuzuki in 
\cite{tsuzuki}: 
Let $\Rep_{K^{\sigma}}(\pi_1(X))'$ be the subcategory of 
$\Rep_{K^{\sigma}}(\pi_1(X))$ consisting of the representations $\rho$ 
such that, for some 
finite morphism $\varphi: \ol{Y} \lra \ol{X}$ with $Y := 
\varphi^{-1}(X) \lra X$ finite 
etale, the restriction of $\rho$ to $I_v$ is trivial 
for any discrete valuation $v$ of 
$k(Y)$ centered on $\ol{Y} \setminus Y$. 
(He calls such a representation `potentially unramified'.) 
Then he has shown that the functor $G$ 
of Crew restricts to the equivalence of categories 
$$ \Rep_{K^{\sigma}}(\pi(X))' \os{=}{\lra} \FIsocd(X,\ol{X})^{\circ}. $$
Note that we have the following proposition, which implies that 
Kedlaya's result is actually a generalization of Tsuzuki's one: 

\begin{proposition}
With the above notation $($with $X$ of arbitrary dimension$)$, 
we have $\Rep_{K^{\sigma}}^S(\pi_1(X)) = \Rep_{K^{\sigma}}(\pi_1(X))'$.  
\end{proposition}

\begin{proof}
First let $\rho$ be an object in $\Rep_{K^{\sigma}}(\pi_1(X))'$. Let us 
take $\varphi: \ol{Y} \lra \ol{X}$, $Y$ as above and let 
$v \in S$. Let $v'$ be a valuation which 
extends $v$ (hence centered on $\ol{Y} \setminus Y$). Then $\rho|_{I_{v'}}$ is trivial. Since $I_v/I_{v'}$ is finite, 
we see that $\rho |_{I_v}$ has finite image and so 
$\rho \in \Rep_{K^{\sigma}}^S(\pi_1(X))$. \par 
On the other hand, let $\rho$ be an object in 
$\Rep_{K^{\sigma}}^S(\pi_1(X))$. 
We may assume that 
$\rho$ has the form $\pi_1(X) \lra GL_n(O_{K^{\sigma}})$ 
(where $O_{K^{\sigma}}$ is the ring of integers of $K^{\sigma}$). Let 
$\ol{\rho}$ be the composite 
$\pi_1(X) \lra GL_n(O_{K^{\sigma}}) \twoheadrightarrow 
GL_n(O_{K^{\sigma}}/2pO_{K^{\sigma}})$ and let $\varphi: Y \lra X$ be the 
finite etale morphism corresponding to the subgroup 
$\Ker (\ol{\rho}) \subseteq \pi_1(X)$. Let 
$\ol{Y} \lra \ol{X}$ be the normalization of $\ol{X}$ in $k(Y)$, which 
will be also denoted by $\varphi$. Let $v$ be a discrete valuation 
of $k(Y)$ centered on $\ol{Y} \setminus Y$. Then, by definition of $\rho$, 
$\rho |_{I_{v |_{k(X)}}}$ has finite image. Hence so is 
$\rho |_{I_{v}}$. On the other hand, the image of 
$\rho |_{I_{v}}$ is contained in 
$\Ker(GL_n(O_{K^{\sigma}}) \ra GL_n(O_{K^{\sigma}}/2pO_{K^{\sigma}}))$ and 
$\Ker(GL_n(O_{K^{\sigma}}) \ra GL_n(O_{K^{\sigma}}/2pO_{K^{\sigma}}))$
contains no nontrivial finite subgroup. So we can conclude that 
$\rho |_{I_{v}}$ is trivial. Hence we have 
$\rho \in \Rep_{K^{\sigma}}(\pi_1(X))'$. 
\end{proof}

One drawback in the above-mentioned result of Kedlaya 
in higher-dimensional case 
is that the number of valuations which we should look at is infinite.  
In this section, we give an alternative formulation of 
`the subcategory of $\Rep_{K^{\sigma}}(\pi_1(X))$ 
with finite local monodromy' 
which looks at only finitely many 
discrete valuations of $k(X)$ and which is 
still equivalent (via $G$) to $\FIsoc(X,\ol{X})^{\circ}$. 
To do so, let us take 
a separable alteration $f:\ol{X}' \lra \ol{X}$ such that 
$(X' := f^{-1}(X), \ol{X}')$ is a smooth pair. (Such $f$ exists by 
de Jong's theorem \cite[4.1]{dejong}). 
Let $\ol{X}' \setminus X' = \bigcup_{i=1}^r Y_i$ 
be the decomposition of $\ol{X}' \setminus X'$ 
(with reduced closed subscheme structure) 
into irreducible 
components, let $v_i\,(1 \leq i \leq r)$ be the discrete valuation on 
$k(X')$ corresponding to $Y_i$ and let us put 
$S' := \{v_i |_{k(X)} \,|\, 1 \leq i \leq r\}$. (Then $S'$ is a finite 
subset of $S$.) We define the subcategory of $\Rep_{K^{\sigma}}(\pi_1(X))$ 
with finite local monodromy along $S'$ by 
$$ \Rep_{K^{\sigma}}^{S'}(\pi_1(X)) := 
\{\rho \in \Rep_{K^{\sigma}}(\pi_1(X)) \,|\, 
\forall v \in S', \rho|_{I_v}\text{ has finite image}\}. $$ 
Then the main theorem in this section is as follows: 

\begin{theorem}\label{app}
Let the notations be as above. Then the equivalence $G$ of Crew 
\eqref{creweq} restricts to the equivalence 
$$ \Rep_{K^{\sigma}}^{S'}(\pi(X)) 
\os{=}{\lra} \FIsocd(X,\ol{X})^{\circ}. $$
In particular, we have the equality 
$
\Rep_{K^{\sigma}}^{S}(\pi_1(X)) = \Rep_{K^{\sigma}}^{S'}(\pi_1(X))$. 
\end{theorem}

\begin{proof}
We prove the following two claims: 
\begin{enumerate}
\item $\Rep_{K^{\sigma}}^S(\pi_1(X)) \supseteq 
G^{-1}(\FIsocd(X,\ol{X})^{\circ})$. 
\item $\Rep_{K^{\sigma}}^{S'}(\pi_1(X)) \subseteq 
G^{-1}(\FIsocd(X,\ol{X})^{\circ})$. 
\end{enumerate}
Since we have 
$\Rep_{K^{\sigma}}^{S}(\pi(X)) \subseteq \Rep_{K^{\sigma}}^{S'}(\pi(X))$ by 
definition, the above two claims imply the theorem. \par 
Let us prove (1). (Here we give a proof which is 
different from that in \cite[2.3.7, 2.3.9]{kedlayaswanII}.) 
Let $\cE$ be an object in $\FIsocd(X,\ol{X})^{\circ}$ and 
put $\rho := G^{-1}(\cE)$. By \cite[1.3.1]{tsuzuki}, there exists a 
separable alteration $\varphi: \ol{X}'' \lra \ol{X}$ such that, if we put 
$X'' := \varphi^{-1}(X)$, the restriction of $\cE$ to 
$\FIsocd(X'',\ol{X}'')^{\circ}$ extends to an object in 
$\FIsoc(\ol{X}'')^{\circ}$. This implies (via Crew's equivalence for 
$\ol{X}''$) that the restriction of $\rho$ to $\pi_1(X'')$ factors 
through $\pi_1(\ol{X}'')$. Now let us take a discrete valuation $v$ of 
$k(X)$ centered on $\ol{X} \setminus X$ and let $v'$ an extension of 
$v$ to $k(X'')$ (hence centered on $\ol{X}'' \setminus X''$). 
Let $z$ be the center of $v'$ and let $O_{v'}$ be the valuation ring 
of $k(X')_{v'}$. Then the homomorphism 
\begin{equation}\label{map1}
\pi_1(\Spec k(X'')_{v'}) \lra \pi_1(X'') \lra \pi_1(\ol{X}'')
\end{equation}
factors as 
\begin{equation}\label{map2}
\pi_1(\Spec k(X'')_{v'}) \lra \pi_1(\Spec O_{v'}) \lra 
\pi_1(\Spec \cO_{\ol{X}'',z}) \lra \pi_1(\ol{X}''). 
\end{equation}
Since the restriction $\rho$ to $\pi_1(\Spec k(X'')_{v'})$ factors 
through \eqref{map1}$=$\eqref{map2}, we see that $\rho |_{I_{v'}}$ is 
trivial. Hence $\rho_{I_v}$ is finite and so we have 
$\rho \in \Rep_{K^{\sigma}}^{S}(\pi(X))$. So the proof of (1) is 
finished. \par 
Let us prove (2). 
Let us take $\rho \in \Rep_{K^{\sigma}}^{S'}(\pi_1(X))$. We may assume that 
$\rho$ has the form $\pi_1(X) \lra GL_n(O_{K^{\sigma}})$. 
Let $\ol{\rho}$ be the composite 
$\pi_1(X) \lra GL_n(O_{K^{\sigma}}) \twoheadrightarrow 
GL_n(O_{K^{\sigma}}/2pO_{K^{\sigma}})$ and let $X'' \lra X'$ be the 
finite etale morphism corresponding to the subgroup 
$\Ker (\ol{\rho} |_{\pi_1(X')}) \subseteq \pi_1(X')$. Let 
$g: \ol{X}'' \lra \ol{X}'$ be the normalization of $\ol{X}'$ in $k(X'')$ and 
let us take a separable alteration $h: \ol{X}''' \lra \ol{X}''$ such that 
the pair $(X''':=h^{-1}(X''), \ol{X}''')$ is a smooth pair. Then 
$\rho$ induces the morphism 
\begin{equation}\label{rho-1}
\rho |_{\pi_1(X''')}: \pi_1(X''') \lra 
\Ker(GL_n(O_{K^{\sigma}}) \ra GL_n(O_{K^{\sigma}}/2pO_{K^{\sigma}})). 
\end{equation}
Let us put $\ol{X}''' \setminus X''' = 
\bigcup_{i=1}^a Y'_i \cup \bigcup_{i=a+1}^b Y'_i$, where 
$Y'_i$'s \,$(1 \leq i \leq a)$ are the irreducible components of 
$\ol{X}''' \setminus X'''$ with $\codim (g\circ h(Y'_i), \ol{X}') = 1$ 
and $Y'_i$'s \,$(a+1 \leq i \leq b)$ are the irreducible components of 
$\ol{X}''' \setminus X'''$ with $\codim (g\circ h(Y'_i), \ol{X}') \geq 2$. 
Let us put $Z := g\circ h (\bigcup_{i=a+1}^b Y'_i)$. Then $g\circ h$ 
induces the morphism of smooth pairs 
$g \circ h: (X''', \ol{X}''' \setminus (g \circ h)^{-1}(Z)) \lra 
(X',\ol{X}' \setminus Z)$. For $1 \leq i \leq a$, let $v'_i$ be the 
discrete valuation of $k(X''')$ corresponding to $Y'_i$. Then, by 
definition, $v'_i |_{k(X)} \in S'$. Hence 
$\rho |_{I_{v'_i |_{k(X)}}}$ has finite image and so is 
$\rho |_{I_{v'_i}}$. On the other hand, by \eqref{rho-1}, 
the image of $\rho |_{I_{v'_i}}$ is contained in 
$\Ker(GL_n(O_{K^{\sigma}}) \ra GL_n(O_{K^{\sigma}}/2pO_{K^{\sigma}}))$. 
Since 
$\Ker(GL_n(O_{K^{\sigma}}) \ra GL_n(O_{K^{\sigma}}/2pO_{K^{\sigma}}))$
contains no nontrivial finite subgroup, 
$\rho |_{I_{v'_i}}$ is trivial for any $1 \leq i \leq a$. So 
we see that $\rho |_{\pi_1(X''')}$ factors through 
$\pi_1(\ol{X}''' \setminus (g \circ h)^{-1}(Z))$. 
This implies (via Crew's equivalence) that the restriction of 
$G(\rho) \in \FIsoc(X)^{\circ}$ to $\FIsoc(X''')^{\circ}$ extends to 
$\FIsoc(\ol{X}''' \setminus (g \circ h)^{-1}(Z))^{\circ}$ 
(hence to $\FIsocd(X''', \ol{X}''' \setminus (g \circ h)^{-1}(Z))^{\circ}$). 
Then, by \cite[2.1.2]{caro} (applied to $g \circ h$), 
we see that the restriction of $G(\rho)$ to $\FIsoc(X')^{\circ}$ 
extends to 
$\FIsocd(X', \ol{X}' \setminus Z)^{\circ}$. By Theorem \ref{main}, 
it extends to $\FIsocd(X', \ol{X}')^{\circ}$, since $\codim (Z,\ol{X}') \geq 
2$. Then, again by \cite[2.1.2]{caro} (applied to $f$), we see that 
$G(\rho)$ extends to $\FIsocd(X,\ol{X})^{\circ}$. 
So we have proved (2). 
\end{proof}

\end{document}